\documentclass[a4paper,12pt]{amsart}
\usepackage{amsmath,amssymb}
\usepackage{amsthm}
\usepackage[dvipdfmx]{graphicx}
\usepackage{mathrsfs}
\usepackage[active]{srcltx}
\usepackage{amscd}
\usepackage{cases}
\usepackage{cite}
 \newtheorem{theorem}{Theorem}[section]
 \newtheorem{proposition}[theorem]{Proposition}

 \newtheorem{corollary}[theorem]{Corollary}
\theoremstyle{definition}

 \newtheorem{example}[theorem]{Example}
\numberwithin{equation}{section}

\newcommand{\R}{\boldsymbol{R}}

\newcommand{\A}{\mathcal{A}}
\newcommand{\D}{\mathcal{D}}
\newcommand{\pmt}[1]{{\begin{pmatrix} #1  \end{pmatrix}}}

\newcommand{\hess}{\operatorname{Hess}}
\newcommand{\sgn}{\operatorname{sgn}}

\newcommand{\rank}{\operatorname{rank}}

\renewcommand{\phi}{\varphi}

\renewcommand{\Gamma}{\varGamma}
\newcommand{\ep}{\varepsilon}
\begin{document}
\title{Normal form of $D_4^+$-singularities of
fronts and its applications}
\author{Kentaro Saji}
\date{\today}
\maketitle

\footnote[0]{
Partly supported by the
JSPS KAKENHI Grant Numbers 18K03301, 22KK0034, 22K03312.}

\begin{abstract}
We construct a form of the $D_4^+$-singularity of 
fronts in $\R^3$
which uses coordinate transformation on the source and
isometry on the target.
As an application, we calculate differential geometric 
invariants near the $D_4^+$-singularity,
and give a Gauss-Bonnet type theorem
for fronts allowing to have this singularity.
\end{abstract}

\section{Introduction}
Wave fronts and frontals are surfaces with singularities in $3$-space
which have normal directions even along singularities.
In these decades, there appeared several articles concerning on 
differential geometry of wave fronts and frontals
\cite{dz,frame,hhnsuy,hhnuy,hnuy,ms,msuy,nuy,OTflat,suyfront}.
Generic singularities of fronts in $3$-space
are 
cuspidal edges and swallowtails.
Generic singularities of $1$-parameter families of
fronts in $3$-space are several corank one bifurcations
and $D_4^\pm$-bifurcations in addition to cuspidal edges and swallowtails
\cite[Section 22.1]{agv}.
Here we call the
central singularity of $D_4^\pm$-bifurcation of fronts
the {\it $D_4^\pm$-singularities\/} for simplicity.

Fundamental differential geometric invariants of the cuspidal edge
are defined in \cite{suyfront}.
The terminology {\it normal form} of a singular point, 
that used in this paper, is a parameterization
using by coordinate transformation on the source
and isometry on the target \cite{WE}.
For the purpose of differential geometric investigation of
singularities,
it is not only convenient, but also
indispensable to study higher order invariants.
Higher order invariants of the cuspidal edges are
studied in \cite{nuy}, and in \cite{nuy},
moduli of isometric deformations of 
the cuspidal edge are determined.
In this paper, we give a normal form of the $D_4^+$-singularity.
As an application, we calculate differential geometric 
invariants near the $D_4^+$-singularity,
and give a Gauss-Bonnet type theorem
for fronts allowing to have this singularity.

The precise definition of fronts, and
the $D_4^+$-singularity is given as follows:
The unit cotangent bundle $T^*_1\R^{3}$ of $\R^{3}$ has the
canonical contact structure and can be identified with the unit
tangent bundle $T_1\R^{3}$. Let $\alpha$ denote the canonical
contact form on it. A map $i:M\to T_1\R^{3}$ is said to be 
{\it
isotropic\/} if the pull-back $i^*\alpha$ vanishes
identically, where $M$ is a $2$-manifold. 
If $i$ is an immersion, then the image $i(M)$ is a Legendre
submanifold, and the image of $\pi\circ i$ is called
the {\it wave front set},
where $\pi:T_1\R^{3}\to\R^{3}$ is the canonical projection and we
denote it by\/ $W(i)$. Moreover, $i$ is called the 
{\it Legendrian lift\/} of $W(i)$. 
With this framework, we define the notion of
fronts as follows: A map-germ $f:(\R^2,0) \to (\R^{3},0)$ is
called a {\it frontal\/} 
if there exists a unit vector field 
(called {\it unit normal of\/} $f$)
$\nu$ of $\R^{3}$ along $f$
such that
$L=(f,\nu):(\R^2,0)\to (T_1\R^{3},0)$ is
an isotropic map by an identification 
$T_1\R^3 = \R^3 \times S^2$, where $S^2$ is 
the unit sphere in $\R^3$ (cf. \cite{agv}, see also \cite{krsuy}).
A frontal $f$ is a {\it front\/} if the above $L$ can be taken as
an immersion.
A point $q\in (\R^2, 0)$ is a singular point if $f$ is not an
immersion at $q$.
A singular point $p$ of a map $f$ is called the
{\it $D_4^+$-singularity\/}
if the map-germ $f$ at $p$ is $\mathcal{A}$-equivalent to
\begin{equation}
\label{eq:normal0}
(u,v)\mapsto
(u^2-v^2,u^2+v^2,u^3+v^3)
\end{equation}
at $0=(0,0)$ (Two map-germs
$f_1,f_2:(\R^n,0)\to(\R^m,0)$ 
are $\mathcal{A}$-{\it equivalent}\/ if there exist diffeomorphisms
$S:(\R^n,0)\to(\R^n,0)$ and $T:(\R^m,0)\to(\R^m,0)$ such
that $ f_2\circ S=T\circ f_1 $.).
A singular point $p$ of a map $f$ is called the
{\it cuspidal edge\/}
if the map-germ $f$ at $p$ is $\mathcal{A}$-equivalent to
$(u,v)\mapsto(u,v^2,v^3)$ at $0$.
By a direct calculation, if the singular point $p$
of $f$ is the $D_4^+$-singularity, or the cuspidal edge, 
then $f$ at $p$ is a front.

\section{Normal form}
Let $\D\subset C^\infty(m,n)$ be an $\A$-invariant set.
An {\it $SO(n)$-normal form of\/ $\D$},
or shortly {\it normal form\/} is a pair $(\mathcal{N},\Phi)$
consists of
a set of functions (includes constants) $\mathcal{N}$ and
a surjection $\Phi:\mathcal{N}\to\widetilde \D$, where
$\widetilde \D=\D/{\rm Diff}_+(m)\times SO(n)$.
Moreover, for $\widetilde{\D^r}\subset \widetilde{\D}$
which includes $J^r(m,n)$ and
$\mathcal{N}^r\subset \mathcal{N}$, the restriction
$\Phi^r=\Phi|_{\mathcal{N}^r}:\mathcal{N}^r\to\widetilde{\D^r}$ is an injection,
then $(\mathcal{N}^r,\Phi^r)$ is called {\it $r$-degree injective\/ 
$SO(3)$-normal form}.
In this case, all elements in $\mathcal{N}^r$ are geometric invariants.

\begin{theorem}
Let\/ $f:(\R^2,0)\to(\R^3,0)$ be a\/ $C^\infty$-map germ.
We assume that there exists a coordinate system\/ $(u,v)$ such
that\/ 
$f_u=u\phi_1(u,v)$
and\/
$f_v=v\phi_2(u,v)$
for functions\/ $\phi_1,\phi_2$ satisfying\/
$\phi_1(0,0)$ and\/ $\phi_2(0,0)$ are linearly independent.
Then there exist\/
$a\in\R$,
$b_1,b_3,c_1,c_3\in C^\infty(1,1)$,
$b_2,c_2\in C^\infty(2,1)$
$s\in {\rm Diff}_+(2)$ and\/
$A\in SO(3)$ such that 
\begin{align}
\label{eq:normal}
A\circ f(s^{-1}(u,v))
=&\Big(u^2-v^2,a(u^2+v^2)+u^3b_1(u)+u^2v^2b_2(u,v)+v^3b_3(v),\\
&\hspace{50mm}u^3c_1(u)+u^2v^2c_2(u,v)+v^3c_3(v)\Big),
\nonumber
\end{align}
where\/ 
$a\geq0$, $c_1(0)\geq0$, $c_3(0)\geq0$.
\end{theorem}
\begin{proof}
We take a coordinate system $(u,v)$ satisfying the assumption.
Since $\phi_1(0,0)$ and $\phi_2(0,0)$ are linearly independent,
by a rotation on the target, we may assume
$$
\phi_1(0,0)=2k_1(1,a,0),\quad\phi_2(0,0)=2k_2(-1,a,0)
\quad (k_1>0,\  k_2>0,\ a>0).
$$
Then $f$ is written by
$$
f(u,v)=(k_1u^2-k_2v^2+g,k_1au^2+k_2av^2+O(3),O(3)),
$$
where $g$ is a function satisfying $j^2g(0)=0$,
and
$O(n)$ stands for a function whose $n-1$-jet at $0$ vanishes.
Since $f_u|_{u=0}=0$ and $f_v|_{v=0}=0$, we see
$g$ has the form
\begin{equation}\label{eq:nouv}
g(u,v)=u^3g_1(u)+u^2v^2g_2(u,v)+v^3g_3(v).
\end{equation}
By setting 
$$
\tilde u=u\sqrt{k_1+ug_1(u)+v^2g_2(u,v)},\quad
\tilde v=v\sqrt{k_2-vg_3(v)},
$$
we see the first component of $f$ is $u^2-v^2$,
and it holds that
$\{u=0\}=\{\tilde u=0\}$,
$\{v=0\}=\{\tilde v=0\}$,
$\partial/\partial\tilde u=\partial/\partial u$ on the $v$-axis and
$\partial/\partial\tilde v=\partial/\partial v$ on the $u$-axis.
This implies that the coordinate system $(\tilde u,\tilde v)$
satisfies that
$f_{\tilde u}=\tilde u\tilde \phi_1(\tilde u,\tilde v)$
and
$f_{\tilde v}=\tilde v\tilde \phi_2(\tilde u,\tilde v)$
for some functions $\tilde \phi_1,\tilde \phi_2$.
Hence by rewriting $(\tilde u,\tilde v)$ to $(u,v)$, 
and by the same reason on \eqref{eq:nouv},
we see $f$ has the form
$$
f(u,v)=(u^2-v^2,a(u^2+v^2)+b(u,v),
u^3c_1(u)+u^2v^2 c_2(u,v)+v^3c_3(v)),
$$
where $b(u,v)=u^3b_1(u)+u^2v^2 b_2(u,v)+v^3b_3(v)$.
By taking $\pi$-rotation
$u\mapsto -u$,
$v\mapsto -v$ if necessary, one may assume $c_1(0)\geq0$. 
If $c_1(0)\geq0$ and $c_3(0)\leq0$, then
we consider the $3\pi/2$-rotation $u\mapsto -v$,
$v\mapsto u$, we see
$f(-v,u)=(-u^2+v^2,a(u^2+v^2)+b(-v,u),
-v^3c_1(-v)+u^2v^2 c_2(-v,u)+u^3c_3(u))$.
By taking the $\pi$-rotation in $\R^3$ with respect to $(0,1,0)$,
we see $(u^2-v^2,a(u^2+v^2)+b(-v,u),
-u^3c_3(u)+u^2v^2 c_2(-v,u)+v^3c_1(-v))$.
Thus we see $c_1(0)\geq0$ and $c_3(0)\geq0$.
This completes the proof.
\end{proof}
We set $f_N$ the right hand side of 
\eqref{eq:normal}, and
$
b(u,v)=u^3b_1(u)+u^2v^2b_2(u,v)+v^3b_3(v)$
$c(u,v)=u^3c_1(u)+u^2v^2c_2(u,v)+v^3c_3(v)$,
$b_{i0}=b_i(0)$,
$c_{i0}=c_i(0)$ $(i=1,3)$.
\begin{theorem}\label{thm:unique}
If\/ $a>0$, $c_{10}>0$, $c_{30}>0$, then
for any\/ $r\geq 3$, 
the map
$$
\Phi^r:(a,j^rb(0),j^rc(0))\mapsto
j^rf_N(0)
$$
is a bijection into\/ $\widetilde\D^r=\D^r/{\rm Diff}_+(2)\times SO(3)$,
where\/ $\D^r$ is the set of\/ $r$-jet who is a germ
of the\/ $D_4^+$-singularity.
\end{theorem}
\begin{proof}
We set $f=f_N$, and assume $A\circ f(s^{-1}(u,v))=f(u,v)$.
We firstly restrict $A\in SO(3)$.
The plane $\nu(0)^\perp$ and the center line must not be changed.
Thus setting $\R^3$ to be the $xyz$-space,
$A$ preserves the $y$-axis including its orientation,
and
$A$ preserves the $x$-axis.
Thus $A$ is the identity or the $\pi$-rotation with respect to the
$y$-axis.
We next see the linear part of $s\in {\rm Diff}_+(2)$.
We set the linear part of $s$ as
$s^{-1}(u,v)=(a_{11}u+a_{12}v,a_{21}u+a_{22}v)$.
Since
$\pm(u^2-v^2)$ is preserved and $u^2+v^2$ is
preserved, either of the following holds
$$
(a_{11},a_{12},a_{21},a_{22})
=(1,0,0,1),\ (-1,0,0,-1),\ (0,1,-1,0),\ (0,-1,1,0).
$$
Then $A\circ f(s^{-1}(u,v))$ is
\begin{align*}
&(u^2-v^2,a (u^2+v^2)+b_{10} u^3+b_{30} v^3 ,c_{10} u^3+c_{30} v^3  ),\\
&(u^2-v^2,a (u^2+v^2)-b_{10} u^3-b_{30} v^3,-c_{10} u^3-c_{30} v^3),\\
&(u^2-v^2,a (u^2+v^2)-b_{30} u^3+b_{10} v^3,c_{30} u^3-c_{10} v^3 ),\\
&(u^2-v^2,a (u^2+v^2)+b_{30} u^3-b_{10} v^3 ,-c_{30} u^3+c_{10} v^3 ),
\end{align*}
respectively,
where $A$ is the identity in the first two cases and
$A$ is the $\pi$-rotation with respect to the $y$-axis in the
last two cases.
Since $c_{10}>0$ and $c_{30}>0$, we see 
$(a_{11},a_{12},a_{21},a_{22})
=(1,0,0,1)$, and $A$ is the identity.
Next, we set $s^{-1}(u,v)=(u+s_1(u,v),v+s_2(u,v))$, where $s_1$ and $s_2$ are
polynomials whose degrees are greater than or equal to $2$.
Looking the third order terms of the first and the second
components of $f(s^{-1}(u,v))$, we see the second order terms of
$s_1$ and $s_2$ are $0$.
We assume $s_1$ and $s_2$ are
polynomials whose degrees are greater than or equal to $k$ $(3\leq k\leq r)$.
Then looking the $(k+1)$-st order terms of the first and the second
components of $f(s^{-1}(u,v))$,
we see the $k$-th order terms of
$s_1$ and $s_2$ are $0$.
This procedure can be continued until $k$ is arbitrary given $r\geq3$,
this completes the proof.
\end{proof}

Let $f$ be a frontal, and let $\nu$ be its unit normal vector field.
We assume that $0$ rank zero singular point.
We set
$$
\lambda=\det(f_u,f_v,\nu)
$$ for some coordinate system $(u,v)$ of $(\R^2,0)$.
\begin{proposition}
The map\/ $f$ given in\/ \eqref{eq:normal} is 
a frontal. Moreover\/ $f$ is a front at\/ $0$ if and only if\/
$c_1(0)c_3(0)\ne0$.
\end{proposition}
\begin{proof}
By the assumption,
we have $f_u=u\phi_1(u,v)$, $f_v=v\phi(u,v)$, where
\begin{align}
\phi_1&=
\dfrac{1}{12}
\Big(1,
12 a+6 u b_1+6 v^2 b_2+2 u^2 (b_1)_u+3 u v^2 (b_2)_u,\nonumber\\
&\hspace{30mm}
6 u c_1+6 v^2 c_2+u (2 u (c_1)_u+3 v^2 (c_2)_u)\Big)\label{eq:phi1}\\
\phi_2&=
\dfrac{1}{12}
\Big(-1,
12 a+6 u^2 b_2+6 v b_3+2 v^2 (b_3)_v+3 u^2 v (b_2)_v,\nonumber\\
&\hspace{30mm}
6 u^2 c_2+v (6 c_3+2 v (c_3)_v+3 u^2 (c_2)_v)\Big).\label{eq:phi2}
\end{align}
Since $a\ne0$, $\phi_1$ and $\phi_2$ are linearly independent,
setting $\tilde\nu=\phi_1\times\phi_2$ and $\nu=\tilde\nu/|\tilde\nu|$,
we see $f$ is a frontal.
Since $df_0=0$ the map $f$ is a front
if and only if $\nu_u$ and $\nu_v$ are linearly independent.
Since 
$$
(\phi_1)_u=\dfrac{1}{2}(0,b_1(0),c_1(0)),\quad
(\phi_1)_v=
(\phi_2)_u=0,\quad
(\phi_2)_v=\dfrac{1}{2}(0,b_3(0),c_3(0)),
$$
and
$$
\nu_u=\dfrac{c_3(0)}{4a}(a,-1,0),\quad
\nu_v=\dfrac{-c_1(0)}{4a}(a,1,0)
$$
hold at $(0,0)$,
we have the assertion.
\end{proof}
\begin{proposition}
Let\/ $f$ be the map\/ $f$ given in\/ \eqref{eq:normal}.
Then the Gaussian curvature\/ $K$ and the mean curvature\/ $H$
satisfy
\begin{align}
\label{eq:gauss}
K&=\dfrac{1}{uv}\bigg(\dfrac{c_1 c_3}{16 a^2}
-\dfrac{3 b_3 c_1^2+9 b_1 c_1 c_3-32 a c_3 (c_1)_u}{192 a^3}u
\nonumber\\
&\hspace{30mm}-\dfrac{9 b_3 c_1 c_3+3 b_1 c_3^2-32 a c_1 (c_3)_v}{192 a^3}v
+O(2)\bigg),\\
H&=\dfrac{1}{uv}\bigg(\dfrac{(1+a^2) c_3}{16 a^2}u
+\dfrac{(1+a^2) c_1}{16 a^2}v
+O(2)\bigg),
\end{align}
where all function values are evaluated at\/ $u=v=0$.
Here, $O(n)$ stands for a function whose\/ $(n-1)$-jet at the origin
vanishes.
\end{proposition}
\begin{proof}
Setting 
$
\tilde E=\phi_1\cdot\phi_1,
\tilde F=\phi_1\cdot\phi_2,
\tilde G=\phi_2\cdot\phi_2
$,
we see $E=u^2\tilde E,
F=uv\tilde F,
G=v^2\tilde G$.
Since $\phi_1$ and $\phi_2$ are linearly independent,
we see $\tilde E\tilde G-\tilde F^2\ne0$.
Moreover,
$(\phi_1)_v=v\phi_3$ and
$(\phi_2)_u=u\phi_3$ 
hold, where
\begin{equation}\label{eq:phi3}
\phi_3=
\dfrac{1}{4}
\Big(0,
4 b_2+2 v (b_2)_v+2 u (b_2)_u+u v (b_2)_{uv},
4 c_2+2 v (c_2)_v+2 u (c_2)_u+u v (c_2)_{uv}\Big).
\end{equation}
Setting
$\tilde L=\det((\phi_1)_u,\phi_1,\phi_2)$,
$\tilde M=\det(\phi_3,\phi_1,\phi_2)$,
$\tilde N=\det((\phi_2)_v,\phi_1,\phi_2)$,
we see
$L=u\tilde L/\delta$,
$M=uv\tilde M/\delta$,
$N=v\tilde N/\delta$,
where $\delta=|\phi_1\times\phi_2|$.
Then
\begin{align}
EG-F^2&=u^2v^2(\tilde E\tilde G-\tilde F^2),\\
LN-M^2&=uv(\tilde L\tilde N-uv\tilde M^2)/\delta^2,\\
EN-2FM+GL&=uv(\tilde E\tilde N-2uv\tilde F\tilde M+v\tilde G\tilde L)/\delta
\end{align}
and
\begin{align}
K&=
\dfrac{\tilde L\tilde N-uv\tilde M^2}
{uv\delta^2(\tilde E\tilde G-\tilde F^2)},\\
H&=
\dfrac{u\tilde E\tilde N-2uv\tilde F\tilde M+v\tilde G\tilde L}
{2uv\delta(\tilde E\tilde G-\tilde F^2)}
\end{align}
hold.
By
\begin{align}
\tilde E\tilde G-\tilde F^2&=
4 a^2+2 a b_1u+2 a b_3v+O(2),\\
\tilde L\tilde N-uv\tilde M^2&=
a^2 c_1 c_3
+\dfrac{1}{12} a (-3 b_3 c_1^2+3 b_1 c_1 c_3+32 a c_3 (c_1)_u)u\nonumber\\
&\hspace{10mm}
+\dfrac{1}{12} a (-3 b_1 c_3^2+3 b_3 c_1 c_3+32 a c_1 (c_3)_v)v+O(2),\\
u\tilde E\tilde N-2uv\tilde F\tilde M+v\tilde G\tilde L&=
a (1+a^2) c_3u+a (1+a^2) c_1v+O(2),\\
\delta&=
2 a+b_1u/2+b_3v/2+O(2),
\end{align}
we have the assertion.
\end{proof}
Let $f$ be the map $f$ given in \eqref{eq:normal}.
Then the set of singular points $S(f)$ satisfies
$S(f)=\{uv=0\}$.
Let us set 
\begin{equation}\label{eq:singpara}
\gamma_1(u)=(u,0),\quad \gamma_2(v)=(0,v)\quad
\text{and}\quad
\hat\gamma_i=f\circ\gamma_i\quad (i=1,2).
\end{equation}
It is known that $f$ at $(u,0)$ or at $(0,v)$ is
cuspidal edge if $u\ne0$ or $v\ne0$.
On cuspidal edge, several geometric invariants are known.
Here we calculate the singular curvature, the normal curvature,
the cuspidal curvature and the cuspidal torsion.
See 
\cite{ms,msuy,suyfront} for definitions and fundamental properties.
\begin{proposition}\label{prop:evalinv}
Under the above setting,
we have
\begin{align}
\label{eq:kappas}
\kappa_s=&
\dfrac{1}{u}
\left(\dfrac{b_1}{2 (1+a^2)^{3/2}}\right.\\
&\hspace{10mm}
\left.+
\dfrac{
-18 a^2 b_1^2+3 c_1^2-3 a^4 c_1^2+32 a (b_1)_u+32 a^3 (b_1)_u}
{24 a (1+a^2)^{5/2}}u+O(2)\right),\label{eq:ksassymp}\\
\kappa_\nu=&
\dfrac{1}{u}
\left(
1+\dfrac{a b_1}{2+2 a^2}u+O(2)\right),\label{eq:knassymp}\\
\kappa_t=&
\dfrac{1}{u}
\left(
\dfrac{(-1+a^2) c_3}{4a (1+a^2)}
+O(1)\right),\label{eq:ktassymp}\\
\kappa_c=&
\dfrac{b_1}{(1+a^2)^{5/4}}+O(1).\label{eq:kcassymp}
\end{align}
\end{proposition}
\begin{proof}
On the $u$-axis, $f_v=0$. Thus we can take $\partial_v$ as a
null vector field. Since $(uv)_v=u$, 
the sign $\sigma$ is $\sigma=\sgn(u)$.
We have
\begin{equation}\label{eq:g1}
\hat\gamma_1(u)=(u^2/2,a u^2/2+u^3 b_1(u)/6,u^3 c_1(u)/6)
\end{equation}
and
$$
\nu(u,0)=
\pmt{0\\0\\1}
-\dfrac{c_1(0)}{4}\pmt{1\\1/a\\0}u
+
\pmt{(3 b_1 c_1+12 a c_2-16 a (c_1)_u)/(48 a)\\
(3 b_1 c_1-12 a c_2-16 a (c_1)_u)/(48 a^2)\\
-(1+a^2) c_1^2/(32 a^2)}u^2+O(3).
$$
Then we see
\begin{align}
\det(\hat\gamma_1'(u),\hat\gamma_1''(u),\nu(\gamma_1(u)))
=&
u^2\left(
\dfrac{b_1}{2}+\left(\dfrac{-(-1+a^2) c_1^2}{8 a}
+\dfrac{4 (b_1)_u}{3}\right)u+O(2)\right),\\
\hat\gamma_1''(u)\cdot\nu(\gamma_1(u))
=&
u\left(1+a^2+\dfrac{3a b_1}{2} u
+O(2)\right),\\
|\hat\gamma_1'(u)|
=&
|u|\sqrt{1+a^2+a b_1u
+O(2)},
\end{align}
and we have \eqref{eq:ksassymp} and \eqref{eq:knassymp}.
Moreover, by
\begin{align}
&  f_u\cdot f_u \det(f_u, f_{vv}, f_{uvv}) 
- f_u\cdot f_{vv} \det(f_u, f_{vv}, f_{uu})\Big|_{v=0} \\
=&u^3\bigg(a (1-a^2) c_1
+O(1)\bigg)\nonumber\\
 &f_u \cdot f_u (f_u\times f_{vv}\cdot f_u\times f_{vv})\Big|_{v=0} 
\nonumber\\
=&
u^4(4 a^2 (1+a^2)
+O(1)),
\end{align}
and
\begin{equation}
\det(\hat\gamma_1'',\hat\gamma_1''',\nu(\gamma_1))(0,0)
=b_1,\quad
|\hat\gamma_1''(0)|=\sqrt{1+a^2},
\end{equation}
we have \eqref{eq:ktassymp} and \eqref{eq:kcassymp}.
\end{proof}
\begin{corollary}
Under the same assumption of Proposition\/ {\rm \ref{prop:evalinv}},
$\kappa_s$ is bounded if and only if\/ $b_1=0$,
and\/ $\kappa_s$ is\/ $0$ if and only if\/ 
$b_1=3 (-1 + a^2) c_1^2 - 32 a (b_1)_u=0$.
On the other hand, $\kappa_\nu$ is always unbounded.
Furthermore, $\kappa_t$ is bounded if\/ $a=\pm1$ or\/ $c_3=0$,
and\/
$\kappa_c$ is\/ $0$ if and only if\/ $b_1=0$.
\end{corollary}
It should be remarked that the cases
$a=\pm1$ correspond to the singularities of parallel surface
of a right-angled umbilic point. See \cite[p391]{fhpara}.
\begin{example}
Since $b_1$ and $b_3$ control the signs of $\kappa_s$, 
we can obtain different curved cuspidal edges
at the top and bottom of a $D_4^+$-singularity
by interchanging this sign.
If the singular curvature is positive, 
then the cuspidal edge is curved 
into a rounded shape, and
negative, it is curved into a warped shape
\cite[Theorem 1.17]{suyfront}.
Let us set
\begin{align}
\label{eq:posks}
f_1&=
(u^2 - v^2, u^2 + v^2+u^3 + v^3, u^3 + v^3),\\
\label{eq:negks}
f_2&=
(u^2 - v^2, u^2 + v^2+u^3 - v^3, u^3 + v^3).
\end{align}
Then by \eqref{eq:kappas}, we see two cuspidal edges 
at the top of $f_1$ have positive $\kappa_s$,
and two cuspidal edges 
at the bottom of $f_1$ have negative $\kappa_s$.
See Figure \ref{fig:posnegks} left.
On the other hand,
by \eqref{eq:kappas}, we see two cuspidal edges 
at the top of $f_2$ have positive and negative $\kappa_s$,
and two cuspidal edges 
at the bottom of $f_2$ have positive and negative $\kappa_s$.
See Figure \ref{fig:posnegks}, right.
\begin{figure}[ht]
\begin{center}
\includegraphics[width=.3\linewidth]{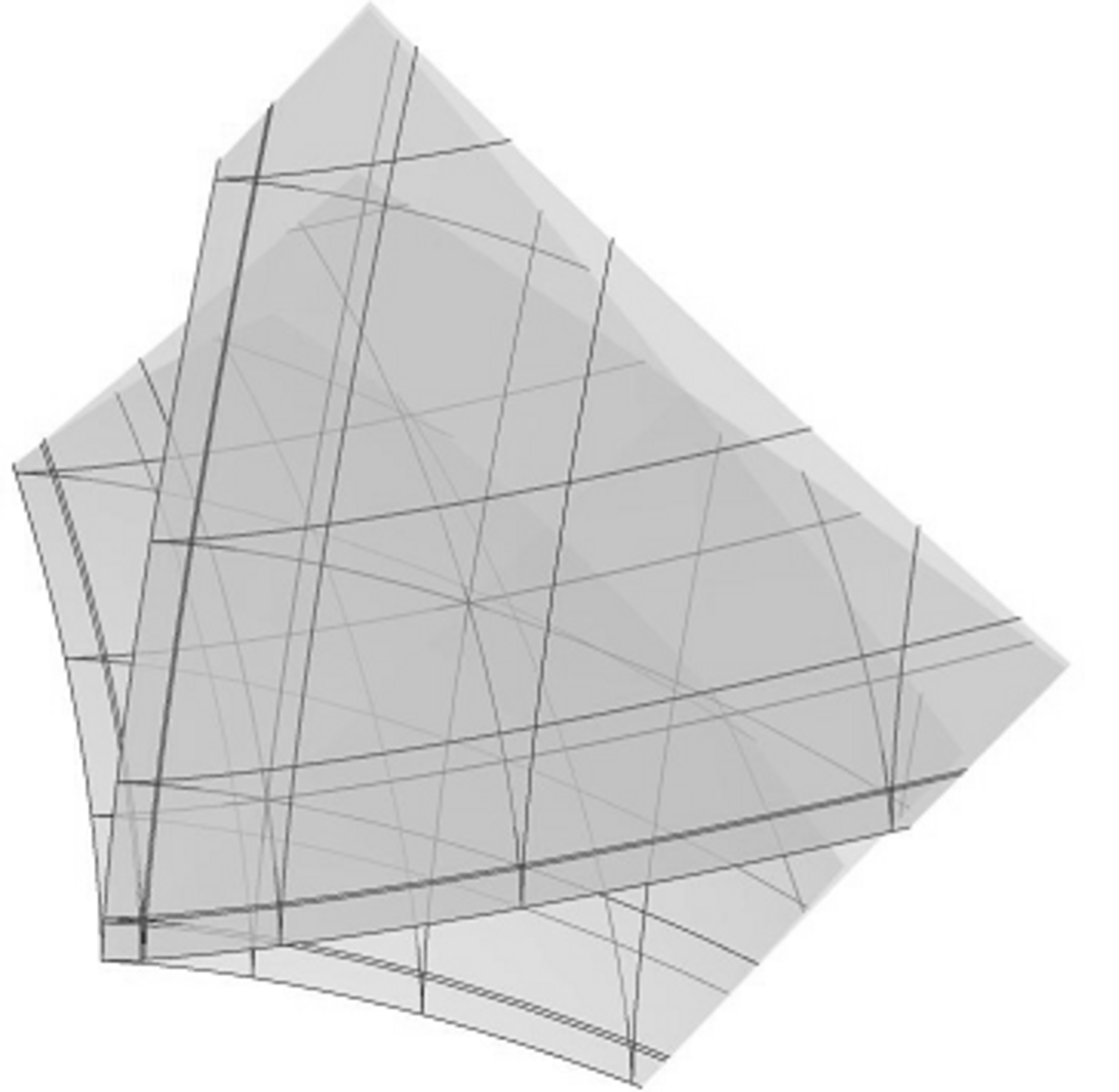}
\hspace{10mm}
\includegraphics[width=.3\linewidth]{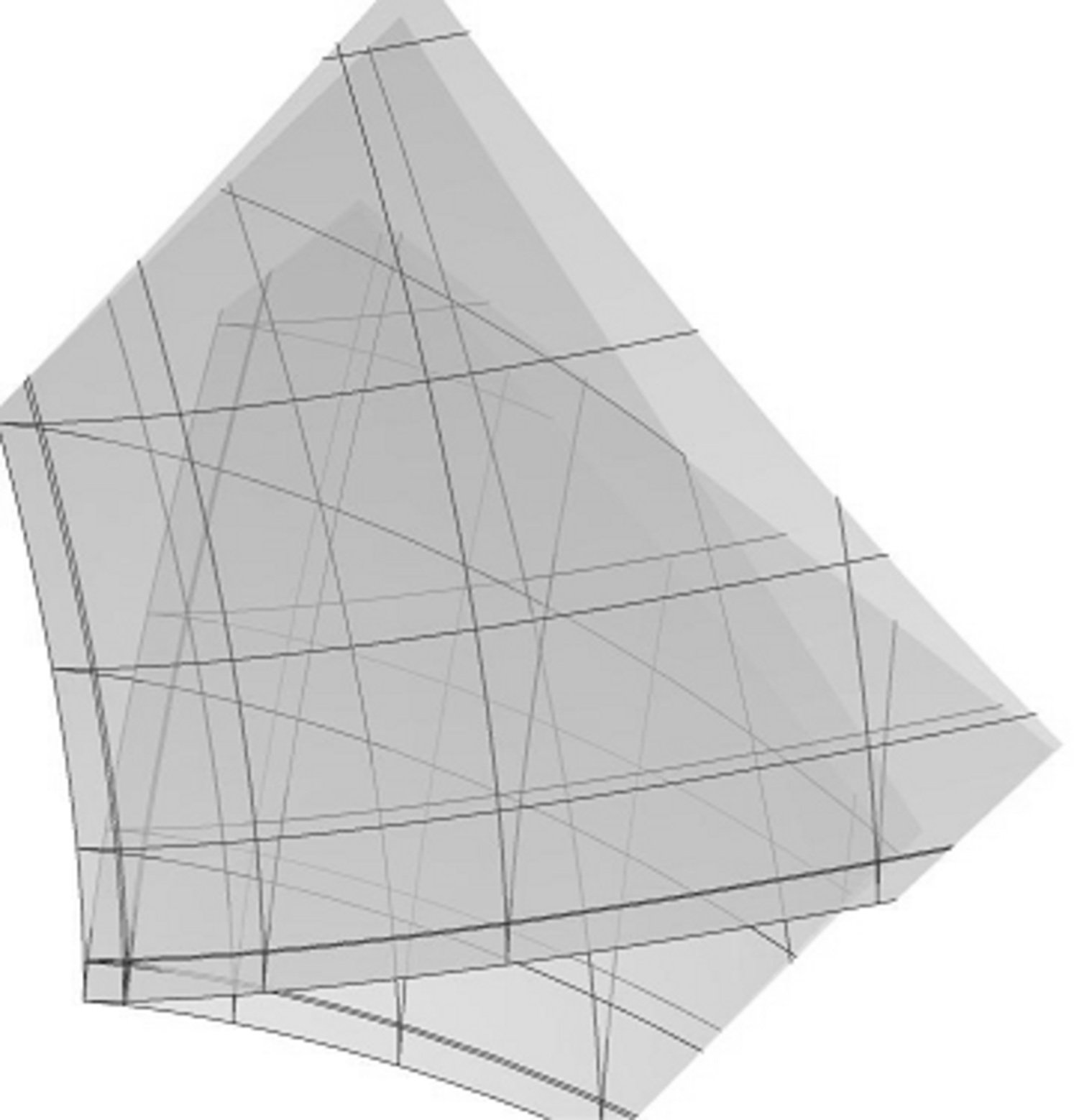}
\end{center}
\caption{$D_4^+$-singularities of \eqref{eq:posks} and \eqref{eq:negks}}
\label{fig:posnegks}
\end{figure}
\end{example}
On the other hand,
we calculate the cuspidal curvature for
plane curves.
Let $f$ be the map $f$ given in \eqref{eq:normal},
and
let $\gamma_i$, $\hat\gamma_i$ $(i=1,2)$ be as in \eqref{eq:singpara}
Let $\nu(0)$ (respectively, $c$) be the unit normal vector
(respectively, unit generator vector of the center line),
and $\pi_{\nu(0)}$ (respectively, $\pi_c$) be orthonormal projections
with respect to $\nu(0)$ (respectively, $c$).
We set 
$\hat\gamma_i^n=\pi_{\nu(0)}\circ f\circ \gamma_i$
and
$\hat\gamma_i^c=\pi_{c}\circ f\circ \gamma_i$.
Let $\Omega(\hat\gamma_i^n)$ 
(respectively, $\Omega(\hat\gamma_i^c)$)
be the cuspidal curvature of $\hat\gamma_i^n$
(respectively, $\hat\gamma_i^c$).
By \eqref{eq:g1},
we have the following corollary.
\begin{corollary}\label{cor:inv}
The angle\/ $\theta$ between\/ 
$\hat\gamma_1''(0)$ and\/
$\hat\gamma_2''(0)$ satisfies 
$$
\cos\theta=\dfrac{-1+a^2}{1+a^2}.
$$
Moreover, it holds that
\begin{align*}
\Omega(\hat\gamma_1^n)=\dfrac{b_1(0)}{1+a^2},\quad
\Omega(\hat\gamma_1^c)=\dfrac{b_3(0)}{1+a^2}\\
\Omega(\hat\gamma_2^n)=\dfrac{c_1(0)}{1+a^2},\quad
\Omega(\hat\gamma_2^c)=\dfrac{c_3(0)}{1+a^2}.
\end{align*}
\end{corollary}
\begin{proof}
By \eqref{eq:g1} and a direct calculation, we have
the assertion for $\gamma_1$.
By the same calculation, one can obtain
the assertion for $\gamma_2$.
\end{proof}
Comparing \eqref{eq:normal}, Corollary \ref{cor:inv}
can be used determining $a,b_i(0),c_i(0)$ $(i=1,3)$ for
a given $D_4^+$-singularity which are all third order
data for $D_4^+$-singularities.
\section{Symmetry}
In this section,
following the method given in \cite{sym},
we discuss symmetries of $D_4^+$-singul{-}arities.
See \cite{sym} for the symmetries of cross caps.
Let $f\in C^\infty(2,3)$ be a $D_4^+$-singularity,
and $\nu$ the unit normal vector field.
We call the plane $\nu^\perp$ through the origin the
{\it tangent plane\/} (see Figure \ref{fig:planes}, left).
By \eqref{eq:g1}, the vectors $\hat\gamma_i''(0)$ $(i=1,2)$ are
linearly independent.
The lines generated by them are called the {\it singular line}.
Moreover, the bisector line of
$\hat\gamma_i''(0)$ $(i=1,2)$ is called the {\it center line\/}
(see Figure \ref{fig:centerline}),
and the normal plane of the center line through the origin the
is called the {\it normal plane\/} (see Figure \ref{fig:planes}, center).
The plane generated by the $\nu(0)$ and center line
is called the {\it principal plane\/} (see Figure \ref{fig:planes}, right).
\begin{figure}[ht]
\begin{center}
\includegraphics[width=.34\linewidth]{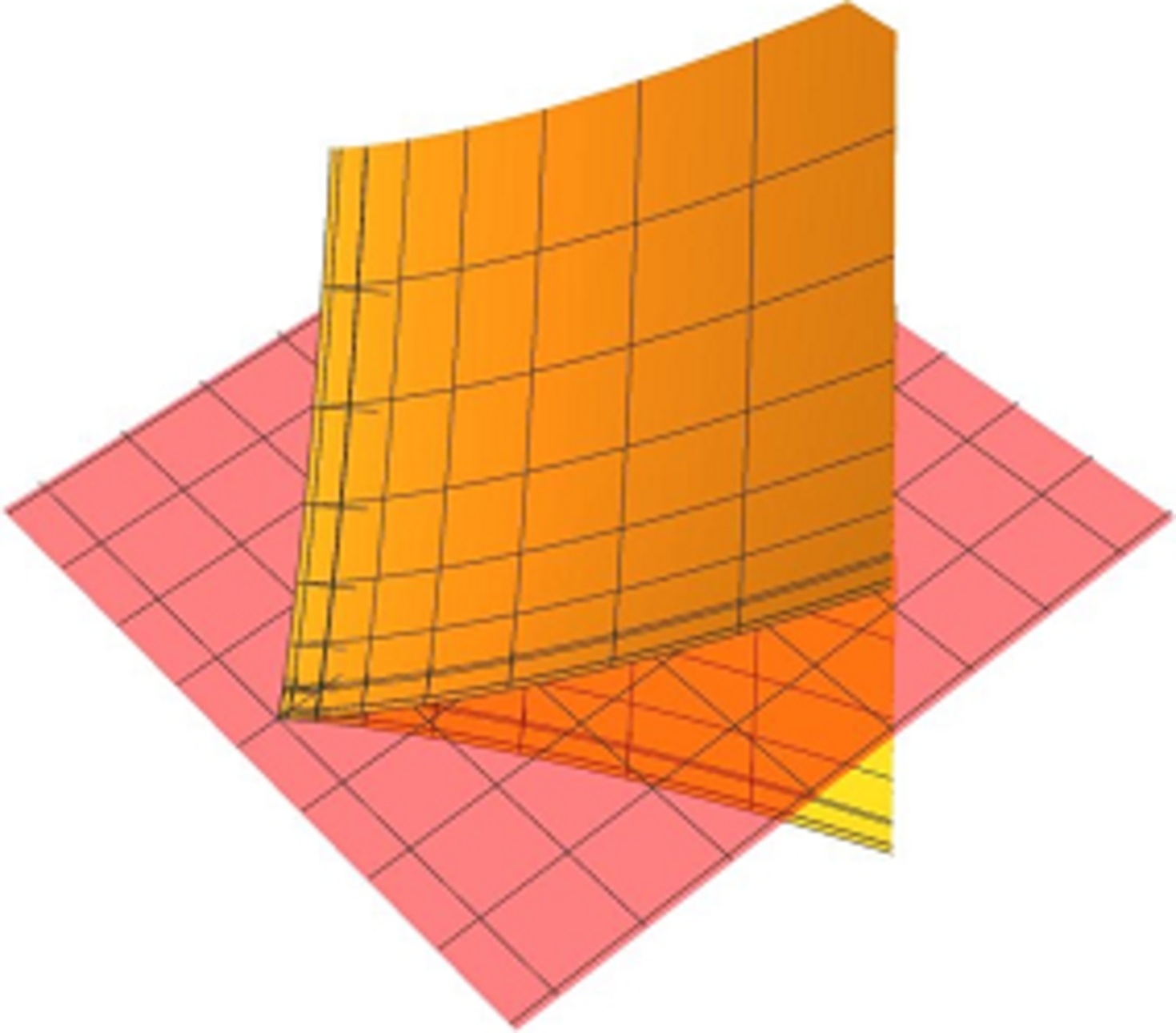}
\hspace{0mm}
\includegraphics[width=.3\linewidth]{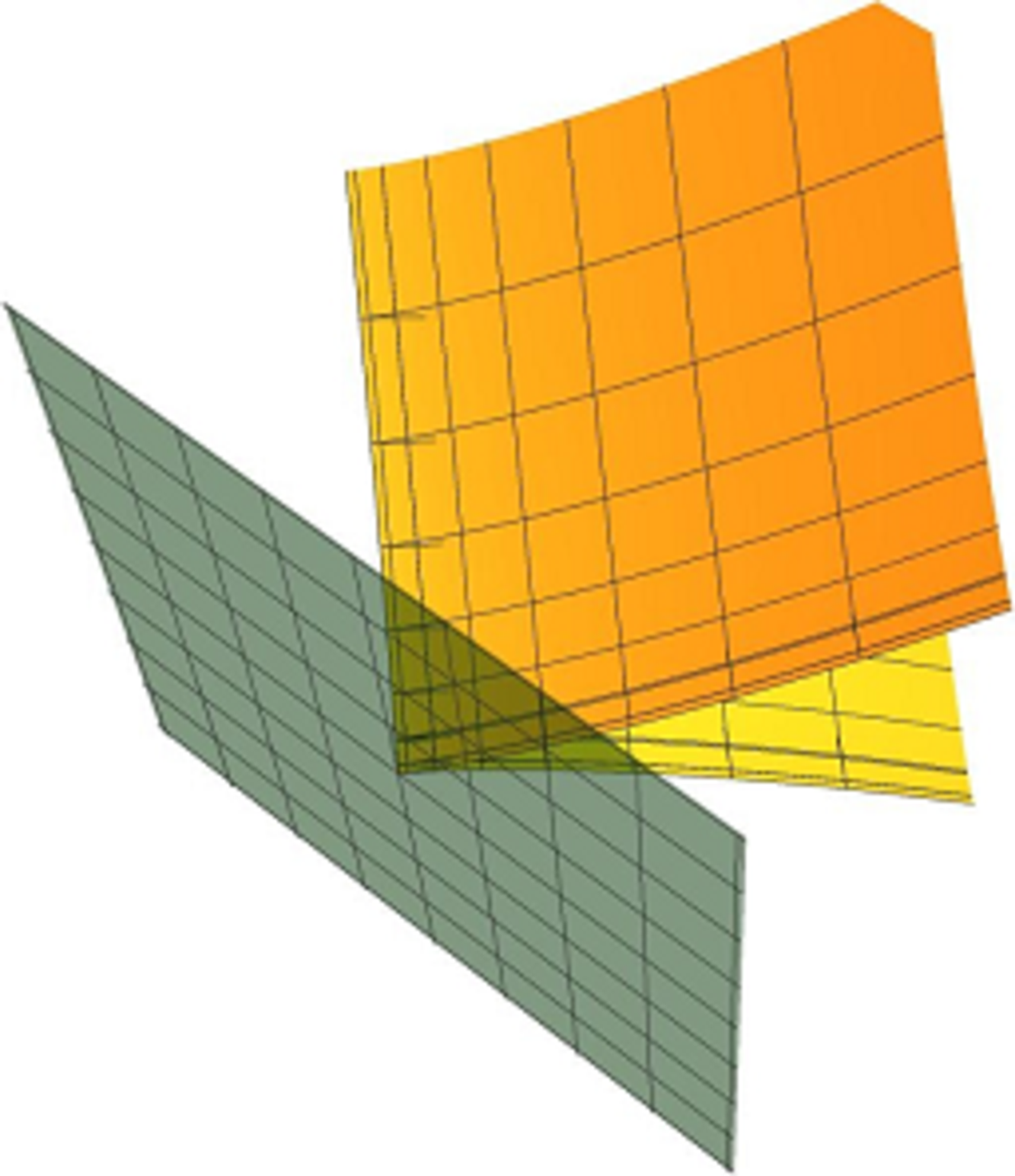}
\hspace{0mm}
\includegraphics[width=.26\linewidth]{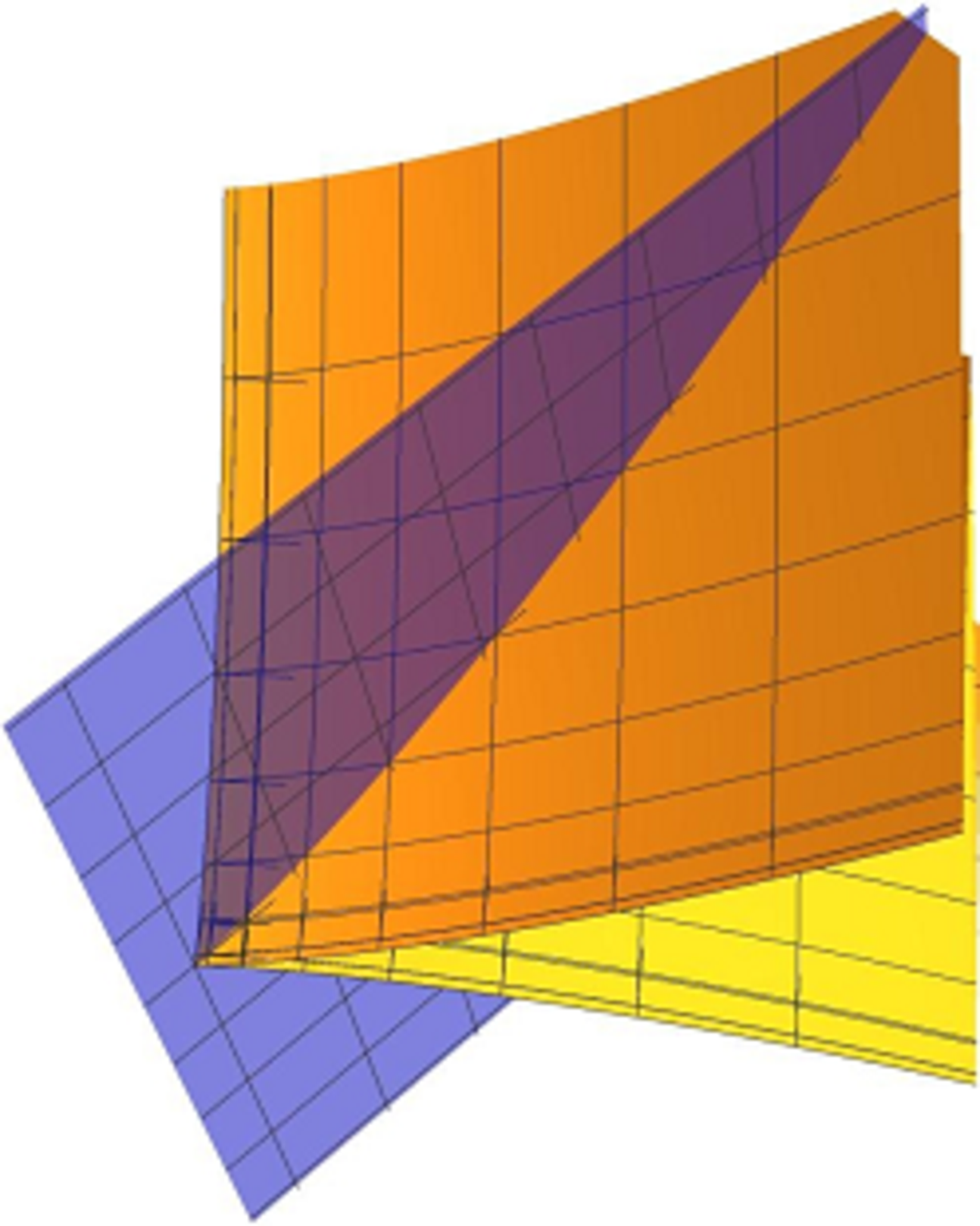}
\end{center}
\caption{The tangent plane, the normal plane and the principal plane}
\label{fig:planes}
\end{figure}
\begin{figure}[ht]
\begin{center}
\includegraphics[width=.3\linewidth]{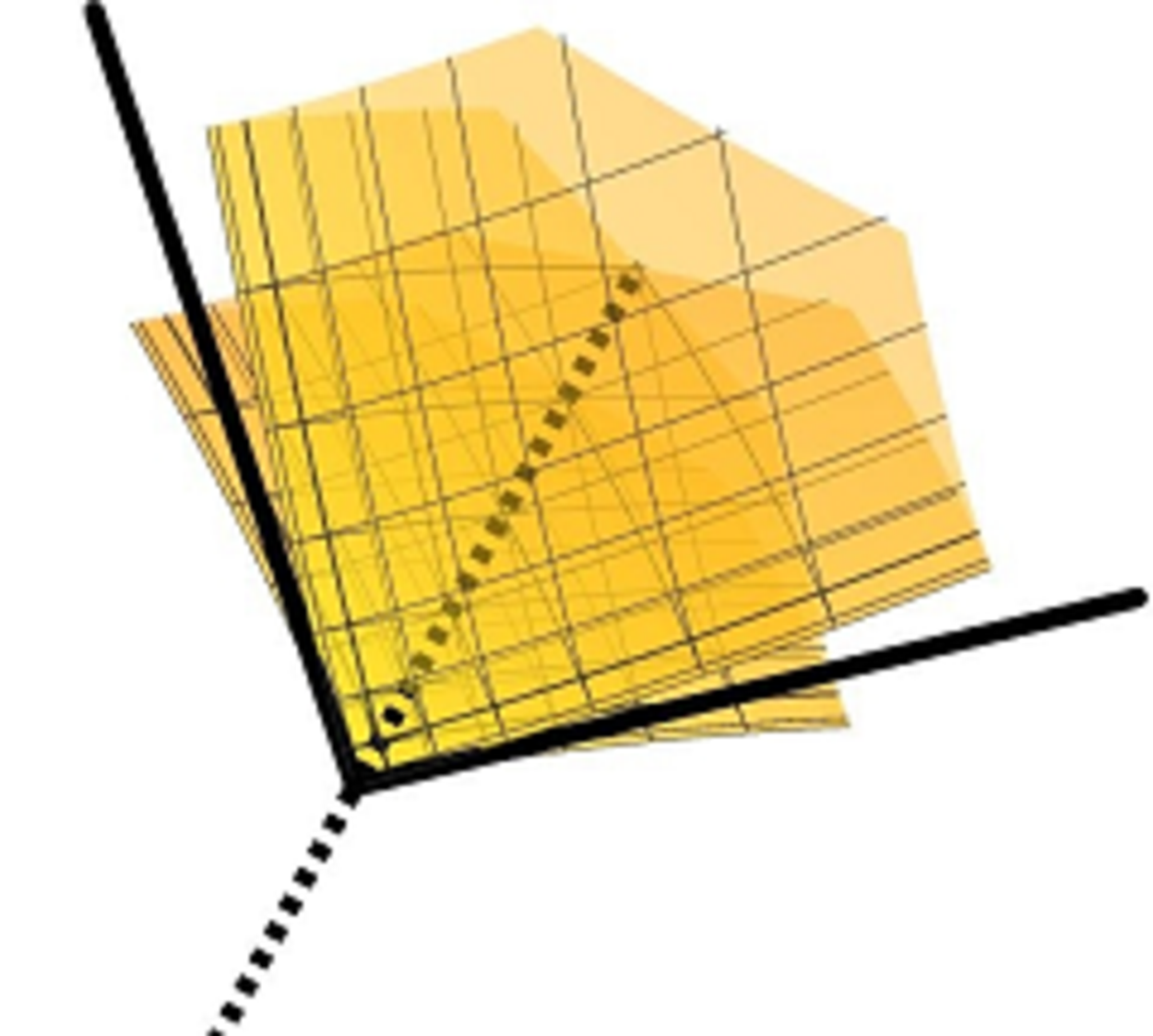}
\end{center}
\caption{The center line (dotted line) and the singular lines (solid lines)}
\label{fig:centerline}
\end{figure}
\begin{theorem}
Let\/ $f\in C^\infty(2,3)$ be a germ of\/ $D_4^+$-singularity.
If there exists\/ $T\in O(2)$ such that\/
$T\circ f(U)=f(U)$ as set germs at\/ $0\in\R^3$ for some open neighborhood
of\/ $0$.
If\/ $T$ is not an identity map, then either
\begin{enumerate}
\item\label{sym:01} a reflection with respect to the tangent plane or,
\item\label{sym:02} a reflection with respect to the principal plane or,
\item\label{sym:03} a\/ $\pi$-rotation with respect to the center line.
\end{enumerate}
If\/ $f\in C^\omega(2,3)$, then
\begin{enumerate}
\setcounter{enumi}{3}
\item\label{sym:11} The above\/ \ref{sym:01} holds 
if and only if\/ $(b(u,v)=b(-u,-v)$, 
$c(u,v)=-c(-u,-v)$,
\item\label{sym:12} The above\/ \ref{sym:02} holds 
if and only if\/ $(b(u,v)=b(v,u)$, 
$c(u,v)=c(v,u))$,
\item\label{sym:13} The above\/ \ref{sym:03} holds 
if and only if\/ $(b(u,v)=b(-v,-u)$, 
$c(u,v)=-c(-v,-u))$.
\end{enumerate}
\end{theorem}
Here, $C^\omega(2,3)$ is the set of analytic-germs $(\R^2,0)\to(\R^3,0)$.
\begin{proof}
By \cite[Theorem A]{zaka}, $T\circ f(U)=f(U)$ implies
that there exists a diffeomorphism-germ $\phi:(\R^2,0)\to(\R^2,0)$
such that $T\circ f\circ\phi=f$.
We remark that the assumption of properness in \cite[Theorem A]{zaka}
is obvious since it does not depend on the $\A$-class and
the \eqref{eq:normal0} satisfies the properness.
In fact, since the second component of \eqref{eq:normal0} is $u^2+v^2$,
thus $u^2+v^2\leq R^2$ implies boundedness of $(u,v)$.
Thus if $f:U\to\R^3$ at $p\in U$ for some $U\subset \R^2$
is a $D_4^+$-singularity, then $f$
is $U$-proper in the sense of \cite{zaka}.
We mentioned in the proof of Theorem \ref{thm:unique},
$T\nu(0)=\pm \nu(0)$ and $Tl=l$ holds.
Thus $T$ is either
$T(1,-1)$, $T(-1,1)$ or $T(-1,-1)$, where
$$T(\ep_1,\ep_2)=
\pmt{\ep_1&0&0\\
0&1&0\\
0&0&\ep_2}\quad(\ep_1=\pm1,\ \ep_2=\pm1).$$
This implies the first assertion.
The second assertion is obvious by considering
$(u,v)\mapsto(\pm u,\pm v)$ together with Theorem \ref{thm:unique}.
\end{proof}
\section{Gauss-Bonnet type theorem}
A map $f:M\to \R^3$ 
between $2$-dimensional closed manifold $M$ and $\R^3$
is called 
a {\it frontal\/} (respectively, a {\it front})
if there exists a map $\nu:M\to S^2$ such that
for any $p\in M$, the map-germ $f$ at $p$
is a frontal (respectively, a front) whose unit normal is $\nu$.
We remark that $\nu$ is defined on $M$.
Gauss-Bonnet type theorems for fronts are obtained in
\cite{suyfront, suykyushu, suycoh}, and it is generalized 
to the case of $\partial M\ne\emptyset$ in \cite{dz}, see also \cite{hashi}.
In these theorems, it is assumed that all singularities $p$ of $f$ 
satisfy $\rank df_p=1$.
Here, we show a Gauss-Bonnet type theorem for
fronts with $D_4^+$-singularities.
We follow the proof of the theorem which
is given in \cite[Section 2]{suyfront} and in \cite[Sections 2,3]{suykyushu},
and then we just need to add considerations of 
boundedness of the Gaussian curvature measure, singular curvature measure
and
the interior angles for each
$D_4^+$-singularity to the proof in \cite{suyfront, suykyushu}.
By \eqref{eq:gauss} and \eqref{eq:kappas},
one can easily see
$K\,dA$ and $\kappa_s\,d\tau$ are bounded measures,
where $dA$ is the area form $dA=\sqrt{EG-F^2}\,du\,dv$ for a coordinate
system $(u,v)$ and $\tau$ is the arclength.
Next we discuss the interior angle of two curves emanating from 
a $D_4^+$-singularity.
Let $f:M\to \R^3$ be a front, and let $f$ at $p$ be a $D_4^+$-singularity.
Let $\gamma:([0,\ep),0)\to(\R^2,0)$ 
be a curve.
We set $\gamma(t)=r(t)(\cos\theta(t),\sin\theta(t))$,
where $r(t)=r_1 t+O(2)$.
Then we see
$f(\gamma(t))=t^2r_1^2(\cos2\theta(0),a/2,0)+O(3)$.
This implies that the initial vector of $f(\gamma(t))$
$$
\lim_{t\to0}
\dfrac{f(\gamma(t))'}{|f(\gamma(t))'|}
$$
is well-defined.
Moreover, we can see
all four interior angles between
two adjacent singular curves emanating from $p$
are $\theta$, where $\theta$ is the angle between
two center lines.
Let $\{q_1,\ldots,q_k\}$ be the set of $D_4^+$-singularities of $f$,
and let $\theta(q_i)$ be the inner angle of two tangent 
lines of $q_i$ $(i=1,\ldots,k)$.
Proving the Gauss-Bonnet type theorem,
taking the triangulation of $M$.
By the above arguments, the proof are 
the same except for the inner angles at $D_4^+$-singular point.
In the usual proof, the sum of the interior angles at 
a vertex of the triangles are $2\pi$,
so at the $D_4^+$-singularities,
the difference between $2\pi$ and $4\theta$ appears.
Thus we have the following claim:
For any $p\in S(f)$, 
the germ $f$ at $p$ is a cuspidal edge, a swallowtail, a peak or a
$D_4^+$-singularities, and
$\{q_1,\ldots,q_k\}$ be the set of $D_4^+$-singularities.
Then it holds that
$$
2\pi\chi(M)
+\sum_{i=1}^k(4\theta(q_i)-2\pi)=\int_MK\,dA+2\int_{S(f)}\kappa_s\,ds.
$$
See \cite[Section 2]{suykyushu} for the definition of peaks.
\section{Focal loci}
Let $f\in U\to\R^3$ $(0\in U\subset \R^2)$ be a map, and let
$x=(x_1,x_2,x_3)\in\R^3$ be a point.
A map $D:U\times\R^3\to\R$ defined by $D(u,v,x)=|x-f(u,v)|^2/2$ is
called the {\it distance squared function},
and it is known that this map plays important roles
to investigate differential geometry of surfaces
from the viewpoint of singularity theory.
See \cite{bg,fh,irrt} for example.
We set $d_x(u,v)=D(u,v,x)$ as a function of two variables.
The {\it focal locus\/} of $f$ at $0$ is the set
$$
\{x\in\R^3\,|\,(d_x)_u=(d_x)_v=\det\hess d_x=0 \text{ at } (u,v)=(0,0)\}.
$$
Focal locus of a map with rank one is well-studied, in particular,
in \cite{fh},
an important notion focal conic is introduced.
In this section, we study that of $f$ given in
the right-hand side of \eqref{eq:normal}
with $a>0$, $c_{10}>0$ and $c_{30}>0$,
which is rank zero.
We see $((d_x)_u,(d_x)_v)(0,0)=(0,0)$,
and
$$
\hess(d_x)(0,0)=\pmt{-2 (x_1+a x_2)&0\\
0&2 (x_1-a x_2)}.
$$
Thus the focal locus of $d_x$ is
intersecting two planes
$
FL=\{x\,|\,x_1=-ax_2\}\cup\{x\,|\,x_1=ax_2\}.
$
Here, we give more precise division of $FL$
according to the singularity of $d_x$ ($x\in FL$).
See \cite[Part II]{agv} for the names and the determinators of
singularities to be dealt with here.
We set the coefficient matrix $C_{m,n}$ of a function 
$h(u,v)=\sum_{i,j=0}^{i=m,j=n} a_{ij} u^iv^j$ at $(0,0)$
by
$$
C_{m,n}(h)=\pmt{
a_{0n}&a_{1n}&\cdots&a_{mn}\\
\vdots&\vdots&\ddots&\vdots\\
a_{00}&a_{10}&\cdots&a_{m0}}.
$$
Let us set $x_2\ne0$ and $x_1=-a x_2$.
Then by a direct calculation, we see
$$
C_{2,3}(d_x)=
\pmt{-2 a x_2&\\
0&0\\
 &0&0&-x_2 b_1(0)-x_3 c_1(0)},
$$
where the blank slots are not necessary for the later calculation.
Then we see $d_x$ at $(u,v)=(0,0)$ is 
$\A$-equivalent to $u^2+v^3$ (the $A_2$-singularity) if and 
only if $x_1=-ax_2$,
$x_3\ne -x_2 b_1(0)/ c_1(0)$ and $x_2\ne0$.
Let us assume $x_1=-ax_2$,
$x_3= -x_2 b_1(0)/ c_1(0)$ and $x_2\ne0$.
Then we have
$$
C_{2,4}(d_x)=
\pmt{-2 a x_2&\\
0&0&0\\
 &0&0&0&\dfrac{1+a^2}{2}+x_2\dfrac{b_1(0) c_1'(0)-b_1'(0)c_1(0)}{c_1(0)}}.
$$
This implies 
$d_x$ is $\A$-equivalent to $u^2\pm v^4$ (the $A_3$-singularity) 
under the above condition if and only if
\begin{equation}\label{eq:x2100}
2\big(c_1(0)b_1'(0)-b_1(0)c_1'(0)\big)x_2\ne
(1+a^2)c_1(0)
\end{equation}
Let us assume $x_1=-ax_2$,
$x_3= -x_2 b_1(0)/ c_1(0)$, $c_1(0)b_1'(0)-b_1(0)c_1'(0)\ne0$ and
$x_2=(1+a^2)c_1(0)/\big(c_1(0)b_1'(0)-b_1(0)c_1'(0)\big)$.
Then
\begin{align*}
C_{2,5}(d_x)&=
\pmt{-2 a x_2&\\
0&0&0\\
 &0&0&0&0&A},\\
A&=\left.
\dfrac{
(1+a^2)(b_1 c_1''-c_1 b_1'')-4 a b_1(b_1 c_1'-c_1 b_1')
}{4(c_1b_1'-b_1c_1')}\right|_{(u,v)=(0,0)}.
\end{align*}
This implies 
$d_x$ is $\A$-equivalent to $u^2+ v^5$ (the $A_4$-singularity) 
under the above condition if and only if
$A\ne0$.
Let us set $x_2\ne0$ and $x_1=a x_2$.
By the same calculation, 
we see the conditions that $d_x$ is 
$\A$-equivalent to the $A_k$-singularity $(k=2,3,4)$
are the same
as the above conditions changing $(b_1,c_1)$ to $(b_3,c_3)$.
Let us set $x_1=x_2=0$.
Then by
$$
C_{3,3}(d_x)=
\pmt{-x_3 c_3(0)\\
0&0\\
0&0&0\\
 &0&0&-x_3 c_1(0)},
$$
we see $d_x$ is $\A$-equivalent to $u^3+ v^3$ (the $D_4^+$-singularity of function)
if $x_3\ne0$.
Let us set $x_1=x_2=x_3=0$.
Then by
$$
C_{4,4}(d_x)=
\pmt{
\dfrac{1+a^2}{2}\\
0&0\\
0&0&-1+a^2\\
0&0&0&0\\
 &0&0&0&\dfrac{1+a^2}{2}},
$$
and the fact that
$2(-1 + a^2)/(1 + a^2)$ never equal to $\pm2$,
we see $d_x$ is $\A$-equivalent to $u^4+ \alpha u^2v^2+v^4$ $(\alpha\ne\pm2)$
($X_9$-singularity).
\begin{example}\label{ex:focal}
Let us set
$$
f(u,v)=
(u^2 - v^2, 
 2 (u^2 + v^2) + u^3 + v^3 (1 + 4 v), 
 u^3 (1 +  u) + v^3 (1 + v)).
$$
Then if $x=(x_1,x_2,x_3)$ satisfies
\begin{itemize}
\item
$x_1=\pm 2 x_2$, $x_3\ne -x_2$ and $x_2\ne0$,
then $d_x$ at $(0,0)$ is the $A_2$-singularity,
\item
$x_1=-2 x_2$, $x_3=-x_2$ and $x_2\ne 5/6$
then $d_x$ at $(0,0)$ is the $A_3$-singularity,
\item
$x_1=-2 x_2$, $x_3=-x_2$ and $x_2= 5/6$
then $d_x$ at $(0,0)$ is the $A_4$-singularity,
\item
$x_1=x_2=0$ and $x_3\ne0$,
then $d_x$ at $(0,0)$ is the $D_4^+$-singularity, and
\item
$x_1=x_2=x_3=0$,
then $d_x$ at $(0,0)$ is the $X_9$-singularity.
\end{itemize}
See Figure \ref{fig:focal} for the configurations of focal set
and the image of $f$.
\begin{figure}[ht]
\begin{center}
\includegraphics[width=.5\linewidth]{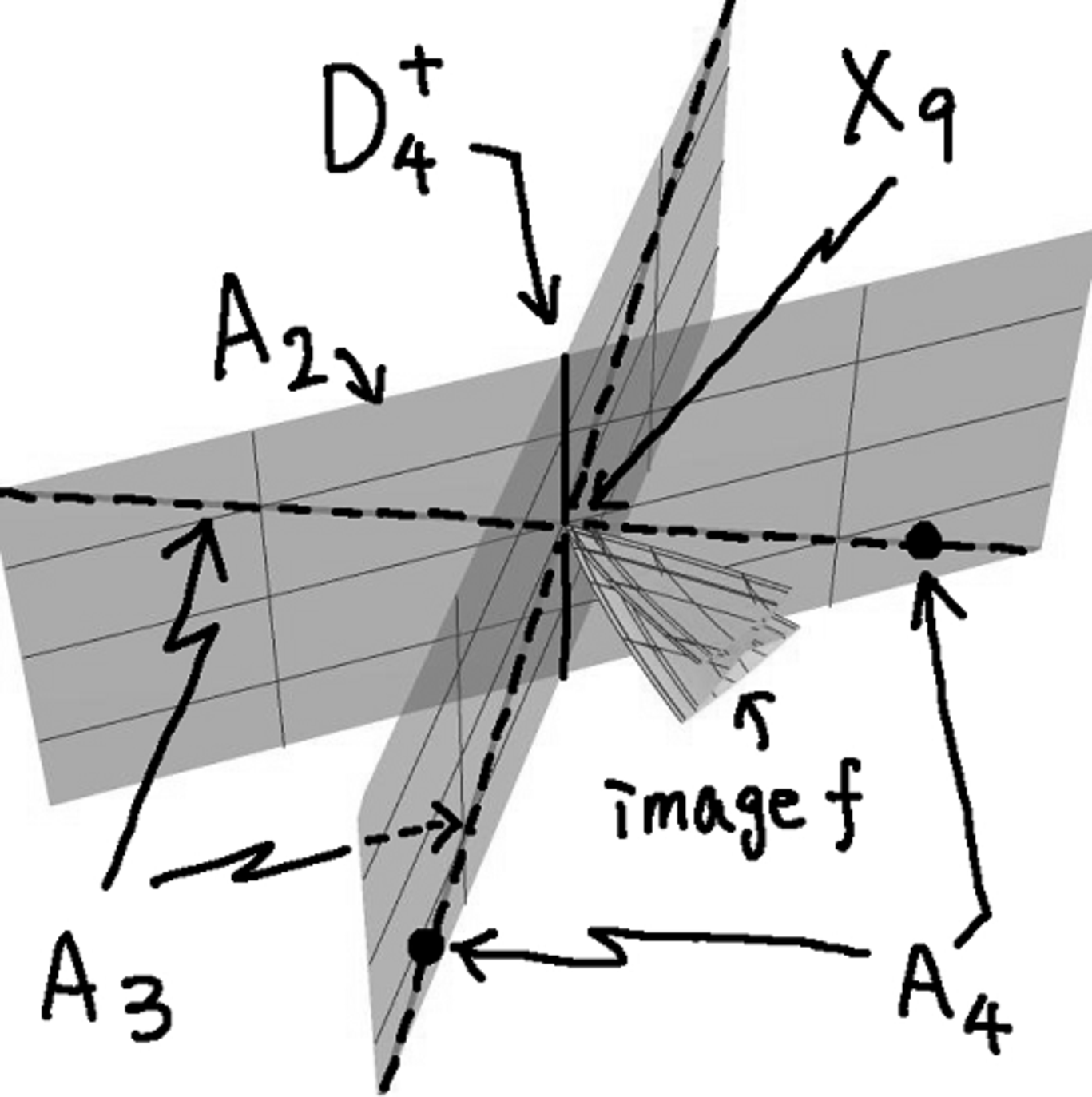}
\end{center}
\caption{Focal set of $f$ in Example \ref{ex:focal}}
\label{fig:focal}
\end{figure}
\end{example}


\medskip
{\footnotesize
\begin{flushright}
\begin{tabular}{l}
Department of Mathematics,\\
Graduate School of Science, \\
Kobe University, \\
Rokkodai 1-1, Nada, Kobe \\
657-8501, Japan\\
E-mail: {\tt saji@math.kobe-u.ac.jp}
\end{tabular}
\end{flushright}}

\end{document}